\documentclass{svmult}

\usepackage{mathptmx}   
\usepackage{helvet}         
\usepackage{courier}        

\usepackage{makeidx}     
\usepackage{graphicx}     
\usepackage{multicol}      
\usepackage[bottom]{footmisc} 

\usepackage{latexsym}
\usepackage{amsmath}
\usepackage{amsfonts}
\usepackage{amssymb}

\usepackage[all]{xy}
\SelectTips{cm}{}

\newtheorem{thm}{Theorem}[section]
\newtheorem{coro}[thm]{Corollary}
\newtheorem{prop}[thm]{Proposition}

\newtheorem{defn}[thm]{Definition}
\newtheorem{rem}[thm]{Remark}

\DeclareRobustCommand{\coprod}{\mathop{\text{\fakecoprod}}}
\newcommand{\fakecoprod}{\sbox0{$\prod$}%
\smash{\raisebox{\dimexpr.9625\depth-\dp0}{\scalebox{1}[-1]{$\prod$}}}%
\vphantom{$\prod$}}

\DeclareMathOperator\Ho{Ho}
\DeclareMathOperator\Ind{Ind}
\DeclareMathOperator\colim{colim}
\newcommand\Set{{\mathbf{Set}}}
\newcommand\op{{\mathrm{op}}}
\newcommand\id{{\mathrm{id}}}

\newcommand\ca{\mathcal {A}}
\newcommand\cc{\mathcal {C}}
\newcommand\cd{\mathcal {D}}
\newcommand\ci{\mathcal {I}}

\newcommand\ck{\mathcal {K}}

\newcommand\cm{\mathcal {M}}

\makeindex     

\begin{document}

\title*{Combinatorial homotopy categories}
\author{Carles Casacuberta and Ji\v{r}\'{\i} Rosick\'{y}}
\institute{Carles Casacuberta \at Institut de Matem\`atica, Universitat de Barcelona (UB), Gran Via de les Corts Catalanes 585, 08007 Barcelona, Spain, \email{carles.casacuberta@ub.edu} \and Ji\v{r}\'{\i} Rosick\'{y} \at Department of Mathematics and Statistics, Masaryk University, Faculty of Sciences, Kotl\'{a}\v{r}sk\'{a} 2, 611 37 Brno, Czech Republic, \email{rosicky@math.muni.cz}}

\maketitle

\abstract{A model category is called combinatorial if it is cofibrantly generated and its underlying category is locally presentable. As shown in recent years, homotopy categories of combinatorial model categories share useful properties, such as being well generated and satisfying a very general form of Ohkawa's theorem.}

\section{Introduction}
\label{sec:1}

The term ``combinatorial'' in topology classically refers to discrete methods or, more specifically, to the use of polyhedra, simplicial complexes or cell complexes in order to deal with topological problems \cite{Kan,W}.

In the context of Quillen model categories in homotopy theory \cite{Quillen}, those called \emph{combinatorial} are, by definition, the cofibrantly generated ones whose underlying category is locally presentable. For example, simplicial sets are combinatorial, but topological spaces are not. As a consequence of this fact, certain constructions involving homotopy colimits, such as Bousfield localizations, may seem intricate if one works with topological spaces while they have become standard technology in the presence of combinatorial models \cite{Barwick, Bousfield, Hirschhorn}.

One key feature of combinatorial model categories is that they admit presentations in terms of generators and relations; in fact, as shown by Dugger in \cite{Dugger}, they are Quillen equivalent to localizations of categories of simplicial presheaves with respect to sets of maps. Moreover, for each combinatorial model category $\ck$ there exist 
cardinals $\lambda$ for which $\ck$ admits fibrant and cofibrant replacement functors that preserve $\lambda$\nobreakdash-fil\-tered colimits and $\lambda$-presentable objects, and the class of weak equivalences is closed under $\lambda$-filtered colimits \cite{Beke, Dugger, R2}.

Cofibrantly generated model categories admit weak generators \cite{Hovey, Raptis}. Combinatorial model categories are, in addition, well generated in the sense of 
\cite{Krause, Neeman}. This fact links the study of combinatorial model categories with the theory of triangulated categories in useful ways. For instance, it was shown in \cite{CGR1} that localizing subcategories of triangulated categories with combinatorial models are coreflective assuming a large-cardinal axiom (Vop\v{e}nka's principle), and similarly colocalizing subcategories are reflective. 

In this article we show that a suitably restricted Yoneda embedding \cite{AR, R2} gives a way to implement Ohkawa's argument \cite{Ohkawa} in the homotopy category of any combinatorial model category, not necessarily stable. Ohkawa's original theorem becomes then a special case, since the homotopy category of spectra admits combinatorial models \cite{HSS}. Thus we prove that, if $\ck$ is a pointed strongly $\lambda$-combinatorial model category (see Section~\ref{sec:2} below for details) then there is only a set of distinct kernels of endofunctors $H\colon\ck\to\ck$ preserving $\lambda$-filtered colimits and the zero object. 

This statement (and our method of proof) is a variant of the main result in~\cite{CGR2}, where Ohkawa's theorem was broadly generalized. In independent work, Stevenson used abelian presheaves over compact objects to prove that Ohkawa's theorem holds in compactly generated tensor triangulated categories \cite{Stevenson}, and Iyengar and Krause extended this result to well generated tensor triangulated categories~\cite{IK}.


\begin{acknowledgement} 
This article has been written as a contribution to the proceedings of a memorial conference for Professor Tetsusuke Ohkawa held at the University of Nagoya in 2015. The content of Section~\ref{sec:4} is based on previous joint work of the authors with Javier Guti\'errez published in \cite{CGR2}. We also appreciate useful discussions with George Raptis.
The authors were supported by the Grant Agency of the Czech Republic under grant P201/12/G028, the Agency for Management of University and Research Grants of Catalonia under project 2014\;SGR\,114, and the Spanish Ministry of Economy and Competitiveness under grant MTM2013-42178-P.
\end{acknowledgement}


\section{Combinatorial model categories}
\label{sec:2}

Recall from \cite{Hovey, Quillen} that if $\ck$ is a model category then its homotopy category $\Ho\ck$ can be defined as the quotient of the full subcategory $\ck_{cf}$ consisting of objects that are fibrant and cofibrant by the homotopy relation on morphisms.
Each choice of a fibrant replacement functor $R_f$ and a cofibrant replacement functor $R_c$ on $\ck$ yields an essentially surjective functor 
\begin{equation}
\label{P}
P\colon\ck\longrightarrow\Ho\ck,
\end{equation}
namely the composite $R_cR_f\colon\ck\to\ck_{cf}$ followed by the 
projection $\ck_{cf}\to\Ho\ck$. 

A model category is called \textit{combinatorial} if it is locally presentable and cofibrantly generated ---the definitions of these terms can be found in \cite{AR, Dugger, Hirschhorn, Hovey}. By a \textit{combinatorial homotopy category} we mean a homotopy category of a combinatorial model category. 

Every locally presentable category $\cc$ can be viewed as a combinatorial homotopy category because  the trivial model structure on $\cc$ (that is, the one in which every morphism is both a cofibration and a fibration, and the weak equivalences are the isomorphisms) is cofibrantly generated by the argument given in \cite[Example~4.6]{R1}. 
In general, combinatorial homotopy categories are far from being locally presentable themselves, but they behave in some sense like a homotopy-theoretical version of those. 

A model category $\ck$ is called \emph{$\lambda$-combinatorial} for a regular cardinal $\lambda$ if it is locally $\lambda$-presentable and cofibrantly generated by morphisms between $\lambda$-presentable objects. Then the functors giving factorizations of morphisms in $\ck$ into cofibrations followed by trivial fibrations and into trivial cofibrations followed by fibrations can be chosen to be $\lambda$-accessible, that is, preserving $\lambda$-filtered colimits. Details are given in \cite[Proposition~3.1]{R2}. 
%
%
%

\section{Restricted Yoneda embedding}
\label{sec:3}

Let $\cc$ be a category and $\ca$ a small full subcategory of $\cc$. 
The \textit{restricted Yoneda embedding}
\[
E_\ca\colon\cc\longrightarrow\Set^{\ca^{\op}}
\]
sends every object $X$ to the hom-set $\cc(-,X)$ restricted to $\ca$. Thus $E_\ca$ is full and faithful on morphisms whose domain is an object of~$\ca$.

The subcategory $\ca$ is called a \emph{generator} of $\cc$ if $E_\ca$ is faithful, and a \emph{strong generator} if $E_\ca$ is faithful and conservative, that is, reflecting isomorphisms. We say that $\ca$ is a \textit{weak generator} if $E_\ca$ reflects isomorphisms whose codomain is the terminal object of~$\cc$. This means that an object of $\cc$ is terminal whenever its image under $E_\ca$ is terminal; hence the objects in a weak generator of $\cc$ form a \emph{left weakly adequate} set in the sense of~\cite{Raptis}.

It was shown in \cite[Theorem~7.3.1]{Hovey} that, if $\ci$ is a set of generating cofibrations in a pointed cofibrantly generated model category $\ck$, then the cofibres of morphisms in $\ci$ form a weak generator of~$\Ho\ck$. The assumption that $\ck$ be pointed can be removed if $\ck$ has a set $\ci$ of generating cofibrations between cofibrant objects, in which case the domains and codomains of morphisms in $\ci$ form a weak generator of $\Ho\ck$, as shown in \cite[Theorem~1.2]{Raptis}.

We also recall that a small full subcategory $\ca$ of a category $\cc$ is called \emph{dense} if every object $X$ in $\cc$ is a colimit of its canonical diagram with respect to~$\ca$. This is equivalent to $E_\ca$ being full and faithful; see \cite[Proposition~1.26]{AR}. Correspondingly, $E_\ca$ is full if and only if $\ca$ is \emph{weakly dense} in the sense that every object $X$
is a weak colimit of its canonical diagram with respect to~$\ca$. Finally, $E_\ca$ is full and conservative if and only if every $X$ is a minimal weak colimit of its canonical diagram with respect to~$\ca$. Recall that a weak colimit $(\delta_d\colon Dd\to X)$ of a diagram $D\colon\cd\to\cc$ is called \emph{minimal} if every morphism $f\colon X\to X$ such that $f\circ\delta_d=\delta_d$ for each $d\in\cd$ is an isomorphism \cite{Christensen}.

\begin{thm}
\label{strongly}
If $\ck$ is a combinatorial model category, then there exist arbitrarily large regular cardinals $\lambda$ such that $\ck$ has the following properties:
\begin{enumerate}
\item[{\rm 1.}] $\ck$ is locally $\lambda$-presentable.
\item[{\rm 2.}] There is a small weak generator of $\Ho\ck$ consisting of $\lambda$-presentable objects.
\item[{\rm 3.}] There are fibrant and cofibrant replacement functors $R_f$ and $R_c$ on $\ck$ that preserve $\lambda$-filtered colimits and $\lambda$\nobreakdash-pres\-entable objects. 
\end{enumerate}
\end{thm}

\begin{proof}
If $\ck$ is combinatorial, then, according to \cite[Corollary~1.2]{Dugger}, there is a zig-zag of Quillen equivalences into another combinatorial model category $\cm$ where all objects are cofibrant. Consequently, the domains and codomains of morphisms in a set of generating cofibrations for $\cm$ form a weak generator of the homotopy category $\Ho\cm$ by \cite[Theorem~1.2]{Raptis}. Since the latter is equivalent to $\Ho\ck$, it follows that $\Ho\ck$ also has a small weak generator $\ca$.

As $\ck$ is locally presentable, there are arbitrarily large regular cardinals $\mu$ such that $\ck$ is locally $\mu$-presentable, by \cite[Theorem~1.20]{AR}. Thus we can choose $\mu$ big enough so that $\ck$ is locally $\mu$-presentable and cofibrantly generated by morphisms between $\mu$-presentable objects, and, furthermore, the objects in the chosen weak generator $\ca$ are $\mu$\nobreakdash-pres\-ent\-able.
Then, as shown in the proof of \cite[Proposition~3.1]{R2}, there are $\mu$-accessible functors giving factorizations of morphisms in $\ck$ into cofibrations followed by trivial fibrations and into trivial cofibrations followed by fibrations. In particular we can pick a fibrant replacement functor $R_f$ and a cofibrant replacement functor $R_c$ that are $\mu$-accessible. Moreover, using \cite[Theorem~2.19]{AR} or \cite[Proposition~7.2]{Dugger}, we can pick a regular cardinal $\lambda\ge\mu$ such that $R_f$ and $R_c$ preserve both $\lambda$-filtered colimits and $\lambda$-presentable objects. 
\qed
\end{proof}

\begin{defn}
{\rm
We call a model category $\ck$ \emph{strongly $\lambda$-combinatorial} if it is combinatorial and $\lambda$ satisfies the conditions stated in Theorem~\ref{strongly}.
}
\end{defn}

For a regular cardinal $\lambda$, let $\ck$ be a strongly $\lambda$-combinatorial model category and denote by $\ck_\lambda$ a small full subcategory of representatives of all isomorphism classes of $\lambda$-presentable objects. Here and in what follows we assume that fibrant and cofibrant replacement functors $R_f$ and $R_c$ have been chosen on $\ck$ so that they preserve $\lambda$-filtered colimits and $\lambda$-present\-able objects.

Let $\Ho\ck_\lambda$ denote the full image of the composition
\[
\xymatrix{
\ck_{\lambda} \lhook\mkern-7mu \ar[r] & \ck  \ar[r]^-{P} & \Ho\ck
}
\]  
where $P$ is $R_cR_f$ followed by projection as in \eqref{P}, and consider the restricted Yoneda embedding
\[
E_\lambda\colon\Ho\ck\longrightarrow\Set^{({\Ho\ck_\lambda})^{\op}}.
\]
Thus the composite $E_{\lambda}P$ preserves $\lambda$-presentable objects.

The next two results follow from \cite[Proposition~5.1 and Corollary~5.2]{R2}.

\begin{thm}
\label{functor}
Let $\ck$ be a strongly $\lambda$-combinatorial model category for a regular cardinal~$\lambda$. Then $E_\lambda P\colon\ck\to\Set^{({\Ho\ck_\lambda})^{\op}}$ 
preserves $\lambda$-filtered colimits.
\end{thm}

\begin{coro}
\label{triangle}
If $\ck$ is strongly $\lambda$-combinatorial, then $E_\lambda P\cong \Ind_\lambda P_\lambda$.
\end{coro}

Here $P_\lambda\colon\ck_{\lambda}\to\Ho\ck_{\lambda}$ is the domain and codomain restriction of $P$, and $\Ind_\lambda$ denotes free completion under $\lambda$-filtered colimits. Therefore $\Ind_\lambda P_{\lambda}$ is a functor from $\ck$ to $\Ind_\lambda\Ho\ck_{\lambda}$.
The statement of Corollary~\ref{triangle} means that $E_\lambda$ factorizes through the inclusion \[
\Ind_\lambda\Ho\ck_\lambda\subseteq\Set^{({\Ho\ck_\lambda})^{\op}},
\]
and the codomain restriction $E_\lambda\colon \Ho\ck\to \Ind_\lambda\Ho\ck_{\lambda}$, which we keep denoting by $E_\lambda$, makes the composite $E_\lambda P$ isomorphic to $\Ind_\lambda P_\lambda$.

If the model category $\ck$ is pointed, then $\Ind_\lambda\Ho\ck_\lambda$ is also pointed and $E_{\lambda}$ preserves the zero object $0$, since $E_{\lambda}0$ is terminal and it is also initial because $0$ is $\lambda$-presentable and $E_{\lambda}$ is full and faithful on morphisms with domain in $\Ho\ck_{\lambda}$.

\begin{coro}
\label{coprod}
If $\ck$ is a strongly $\lambda$-combinatorial model category, then the functor
$E_\lambda\colon\Ho\ck\to\Ind_\lambda\Ho\ck_\lambda$
preserves coproducts.
\end{coro}

\begin{proof}
Pick a cofibrant replacement functor $R_c$ preserving $\lambda$-filtered colimits and $\lambda$-presentable objects.
Note that $P$ preserves coproducts between cofibrant objects and every object in $\Ho\ck$ is isomorphic to $PX$ for some cofibrant object $X$ in~$\ck$. Hence, using Corollary~\ref{triangle} it suffices to show that $\Ind_\lambda P_\lambda$ preserves coproducts between cofibrant objects.
Since each coproduct is a $\lambda$-filtered colimit of $\lambda$-small coproducts and $\Ind_\lambda P_\lambda$ preserves $\lambda$-filtered colimits, we have to prove that $\Ind_\lambda P_\lambda$ preserves $\lambda$-small coproducts between cofibrant objects. Let $\coprod_{i\in I} K_i$ be such a coproduct, so that the cardinality of $I$ is smaller than~$\lambda$. 
Since the functor $R_c$ preserves $\lambda$\nobreakdash-filt\-ered colimits and $\lambda$\nobreakdash-pres\-ent\-able objects, each $K_i$ is a $\lambda$\nobreakdash-filt\-ered colimit of cofibrant $\lambda$\nobreakdash-pres\-ent\-able objects. Let $D_i\colon\cd_i\to\ck_\lambda$ denote the corresponding diagrams, so that $K_i\cong\colim D_i$. Then $\coprod_{i\in I}K_i$ is a colimit of a $\lambda$-filtered diagram whose values are coproducts $\coprod_{i\in I} D_id_i$ with $d_i\in\cd_i$, and each such coproduct $\coprod_{i\in I} D_id_i$ is $\lambda$-presentable as the cardinality of $I$ is smaller than~$\lambda$. Since the functor $\Ind_\lambda P_\lambda$ preserves $\lambda$\nobreakdash-filt\-ered colimits and $P_\lambda$ preserves $\lambda$-small coproducts of cofibrant objects, the result is proved.
\qed
\end{proof}

\begin{defn}
\label{small}
{\rm
Let $\cc$ be a category with coproducts and $\lambda$ a cardinal. An object $S$ of $\cc$ is
\emph{$\lambda$-small} if for every morphism $f\colon S\to\coprod_{i\in I}X_i$ there is a subset $J$ of $I$ of cardinality less than $\lambda$ such that $f$ factorizes as
\[
\xymatrix{
S \ar[r] & \coprod_{j\in J}X_j \ar[r] & \coprod_{i\in I}X_i,
}
\]
where the second morphism is the subcoproduct injection.
}
\end{defn}
 
We also say that $\aleph_0$-small objects are \emph{compact}. This terminology is due to Neeman \cite{Neeman}, who found how compactness should be defined for uncountable cardinals in triangulated categories. His definition was simplified by Krause in~\cite{Krause}. They considered compactness in additive categories but the definition makes sense in general.
Consider classes $\mathcal S$ of $\lambda$-small objects in a category $\cc$ with coproducts such that for every morphism $f\colon S\to\coprod_{i\in I}X_i$ with $S\in\mathcal S$ there exist morphisms $g_i\colon S_i\to X_i$ for which $S_i\in\mathcal S$ for all $i\in I$ and $f$ factorizes through 
\[
\coprod_{i\in I}g_i\colon\coprod_{i\in I}S_i\longrightarrow\coprod_{i\in I}X_i.
\]
Since the collection of such classes is closed under unions, there is a greatest class with this property.  
Its objects are called $\lambda$-\textit{compact}.

\begin{prop}
\label{many}
If $\ck$ is a strongly $\lambda$-combinatorial model category, then all objects in $\Ho\ck_{\lambda}$ are $\lambda$-compact in $\Ho\ck$.
\end{prop}

\begin{proof}
Choose fibrant and cofibrant replacement functors $R_f$ and $R_c$ preserving $\lambda$\nobreakdash-filt\-ered colimits and $\lambda$-presentable objects, and let $P\colon\ck\to\Ho\ck$ be as in~\eqref{P}.
Suppose given a morphism $g\colon PA\to\coprod_{i\in I} PK_i$ in $\Ho\ck$ where $A$ is in $\ck_{\lambda}$.
According to Corollary~\ref{coprod}, we have
\[
E_\lambda g\colon E_\lambda PA\longrightarrow\coprod_{i\in I}E_\lambda PK_i.
\]
Due to the fact that $E_{\lambda}P$ preserves $\lambda$-presentable objects,
$E_\lambda PA$ is $\lambda$-presentable in $\Ind_\lambda\Ho\ck_\lambda$. Since each coproduct is a $\lambda$-filtered co\-limit of $\lambda$-small subcoproducts, $E_\lambda g$ factorizes through some $\coprod_{j\in J}E_\lambda PK_j$ where $J$ has cardinality smaller than~$\lambda$. Since 
$E_\lambda $ is full and faithful on morphisms with domain in $\Ho\ck_\lambda$, we obtain a factorization of $g$ through $\coprod_{j\in J}PK_j$ and therefore we conclude that $PA$ is $\lambda$-small.
 
Moreover, the argument used in the proof of Corollary~\ref{coprod} shows in a similar way that $E_{\lambda}g$ factors through some coproduct $\coprod_{j\in J} E_{\lambda}PD_jd_j$ where $J$ has cardinality smaller than $\lambda$ and $D_jd_j$ is in $\ck_{\lambda}$ for all~$j$. Using again the fact that $E_\lambda $ is full and faithful on morphisms with domain in $\Ho\ck_\lambda$, we find a factorization of $g$ through $\coprod_{j\in J} PD_jd_j$. Hence $PA$ is indeed $\lambda$-compact.
\qed
\end{proof}

\begin{defn}
\label{wellgen}
{\rm
A category with coproducts is called \textit{well $\lambda$-generated} if it has a small weak generator consisting of $\lambda$-compact objects.
It is called \textit{well generated} if it is well $\lambda$-generated for some $\lambda$.
}
\end{defn}

For example, every locally $\lambda$-presentable category is well $\lambda$-generated. 

The following result was proved in \cite[Proposition~6.10]{R2} with the additional assumption that $\ck$ was stable, which is not necessary.

\begin{thm}
\label{well}
If $\ck$ is a strongly $\lambda$-combinatorial model category, then $\Ho\ck$ is well $\lambda$-generated.  
\end{thm}

\begin{proof}
Since, by assumption, there is a small weak generator of $\Ho\ck$ whose objects are $\lambda$-pres\-ent\-able, $\Ho\ck_\lambda$ weakly generates $\Ho\ck$.  
The rest follows from Proposition~\ref{many}.
\qed 
\end{proof} 

As a corollary one infers Neeman's result in \cite{N1} that, for any Grothendieck abelian category $\ca$, the derived category $D(\ca)$ is well generated.  


\section{Ohkawa's theorem}
\label{sec:4}

For an endofunctor $H\colon\ck\to\ck$ (not necessarily preserving weak equivalences) on a model category~$\ck$, we consider the composition
\[
\xymatrix{
\ck \ar[r]^-{H} & \ck \ar[r]^-{P} & \Ho\ck,
}
\]
where $P$ is defined as in \eqref{P}.
The class of objects $X$ in $\ck$ such that $PHX$ is the terminal object in $\Ho\ck$ will be called the \emph{kernel} of $H$ and will be denoted by $\ker H$. Hence, if $\ck$ is pointed and $0$ denotes the zero object in $\ck$ and also its image in $\Ho\ck$, then $\ker H$ consists of objects $X$ in $\ck$ such that $PHX=0$.

In this section we prove the following result.

\begin{thm}
\label{ohkawa}
Suppose that $\ck$ is a pointed strongly $\lambda$-combinatorial model category.
Then there is only a set of distinct kernels of endofunctors $H\colon\ck\to\ck$ preserving $\lambda$-filtered colimits and the zero object.
\end{thm}

\begin{proof}
Consider the restricted Yoneda embedding as given by Corollary~\ref{triangle},
\[
E_{\lambda}\colon\Ho\ck\longrightarrow\Ind_\lambda\Ho\ck_\lambda.
\]
For a morphism $f\colon E_{\lambda}S\to E_{\lambda}PHA$ with $A\in\ck_{\lambda}$ and $S\in\Ho\ck_{\lambda}$, let us denote by $T_H(f)$ the set of all morphisms $t\colon A\to B$ in $\ck_{\lambda}$ such that the composite
\[
\xymatrix{
E_{\lambda}S \ar[r]^-{f} & E_{\lambda}PHA \ar[rr]^{E_{\lambda}PHt} & & E_{\lambda}PHB
}
\]
is the zero morphism, that is, $E_{\lambda}PHt\circ f$ factors through the zero object.

Next, we denote
\[
J(H)=\{T_H(f)\mid f\colon \text{$E_{\lambda}S\to E_{\lambda}PHA$ with $A\in\ck_{\lambda}$ and $S\in\Ho\ck_{\lambda}$}\}.
\]

We are going to prove that if $J(H_1)=J(H_2)$ then $\ker H_1=\ker H_2$, assuming that $H_1$ and $H_2$ preserve $\lambda$-filtered colimits and the zero object. Thus suppose that $J(H_2)\subseteq J(H_1)$ and let $X\in\ker H_1$. 
In order to prove that $PH_2X=0$, it is enough to show that every morphism $E_{\lambda}G\to E_{\lambda}PH_2X$ factors through the zero object if $G$ is in $\Ho\ck_{\lambda}$, since $\Ho\ck_{\lambda}$ is a weak generator of $\Ho\ck$ and $E_{\lambda}$ is full and faithful on morphisms whose domain is in $\Ho\ck_{\lambda}$.

Assume given such a morphism $f\colon E_{\lambda}G\to E_{\lambda}PH_2X$. Since the category $\ck$ is locally $\lambda$\nobreakdash-pres\-entable, $X\cong\colim(D\colon\cd\to\ck_{\lambda})$ for a certain $\lambda$-filtered diagram $D$. Since $E_{\lambda}PH_2$ preserves $\lambda$-filtered colimits by Theorem~\ref{functor}, we then have
\[
E_{\lambda}P H_2X\cong\colim\left(\xymatrix{\cd \ar[r]^-{D} & \ck_{\lambda} \ar[r]^-{PH_2} &
\Ho\ck \ar[r]^-{E_{\lambda}} & \Ind_\lambda\Ho\ck_\lambda}\right).
\]

Since $E_{\lambda}G$ is $\lambda$-presentable, $f$ factors through $\hat f\colon E_{\lambda}G\to E_{\lambda}P H_2Dd$ for some $d\in\cd$. Note that the set $T_{H_2}(\hat f)$ is nonempty, since the morphism $Dd\to 0$ is in it as $H_2$ preserves the zero object. Consequently, the assumption that $J(H_2)\subseteq J(H_1)$ implies that $T_{H_2}(\hat f)\in J(H_1)$. This means that there exist an object $V\in\Ho\ck_{\lambda}$ and a morphism $g\colon E_{\lambda}V\to E_{\lambda}PH_1Dd$ such that $T_{H_2}(\hat f)=T_{H_1}(g)$.

Now, since $X\in\ker H_1$, we have $E_{\lambda}PH_1X=0$. However,
\[
E_{\lambda}P H_1X\cong\colim\left(\xymatrix{\cd \ar[r]^-{D} & \ck_{\lambda} \ar[r]^-{PH_1} &
\Ho\ck \ar[r]^-{E_{\lambda}} & \Ind_\lambda\Ho\ck_\lambda}\right),
\]
and, since $E_{\lambda}V$ is $\lambda$-presentable, there is a morphism $\delta\colon d\to d'$ in $\cd$ such that
\[
\xymatrix{
E_{\lambda}V \ar[r]^-{g} & E_{\lambda}PH_1Dd \ar[rr]^-{E_{\lambda}PH_1D\delta} & &
E_{\lambda}PH_1Dd'
}
\]
factors through the zero object. Hence $D\delta\in T_{H_1}(g)$. Therefore $D\delta\in T_{H_2}(\hat f)$ and this implies that $f\colon E_{\lambda}G\to E_{\lambda}PH_2X$ factors through the zero object, as we wanted to show.

Finally, since there is only a set of distinct sets $J(H)$, the theorem is proved.
\qed
\end{proof}

Ohkawa's theorem \cite[Theorem~2]{Ohkawa} is a special case of Theorem~\ref{ohkawa}. Recall that two (reduced) homology theories $E_*$ and $F_*$ on spectra are said to be \emph{Bousfield equivalent} if the class of $E_*$\nobreakdash-acyc\-lic spectra coincides with the class of $F_*$-acyclic spectra. A spectrum $X$ is called \emph{$E_*$-acyclic} if $E_*(X)=0$.

\begin{coro}
There is only a set of Bousfield equivalence classes of representable homology theories on spectra.
\end{coro}

\begin{proof}
The homotopy category of spectra admits a combinatorial model category $\ck$; for instance, symmetric spectra over simplicial sets \cite{HSS}. For each cofibrant spectrum $E$ we consider the endofunctor on $\ck$ defined as $H_EX=E\wedge R_cX$ where $R_c$ is a cofibrant replacement functor preserving filtered colimits. Since smashing with $E$ has a right adjoint, $H_E$ preserves filtered colimits. Moreover, a spectrum $X$ is in $\ker H_E$ if and only if $X$ is $E_*$-acyclic.
Hence Theorem~\ref{ohkawa} implies that there is only a set of distinct kernels of endofunctors of the form~$H_E$.
\qed
\end{proof}

\section{Generalized Brown representability}
\label{sec:5}

In this section we prove other properties of combinatorial homotopy categories related to results in~\cite{R2}.
Note that if $\cc$ is a locally $\lambda$-presentable category with the trivial model structure then the functor $E_\lambda\colon\cc\to\Ind_\lambda\cc_\lambda$ is an isomorphism.

\begin{defn}
\label{brown}
{\rm
A strongly $\lambda$-combinatorial model category $\ck$ is called \textit{$\lambda$-Brown on morphisms} if $E_\lambda\colon\Ho\ck\to\Ind_\lambda\Ho\ck_\lambda$ is full. It is called
\textit{$\lambda$-Brown on objects} if $E_\lambda$ is essentially surjective. Finally, $\ck$ is called \textit{$\lambda$-Brown}
if it is $\lambda$-Brown both on objects and on morphisms.
}
\end{defn}

Let us remark the following facts:
\begin{enumerate}
\item[(i)] A locally finitely presentable stable combinatorial model category is $\omega$-Brown if it is Brown in the sense of \cite{HPS}, where $\omega$ denotes the first infinite ordinal.
\item[(ii)] Whenever $\ck$ is strongly $\omega$-combinatorial and $E_\omega$ is full then $E_\omega$ is essentially surjective as well. 
In fact, by Corollary~\ref{triangle}, $\Ind_\omega P_\omega$ is full; since each object of $\Ind_\omega \ck_\omega$ can be obtained by taking successive colimits of smooth chains \cite{AR} and $P_\omega$ is essentially surjective on objects, $\Ind_\omega P_\omega$ is essentially surjective on objects too. Hence $\ck$ is $\omega$-Brown on objects. This argument does not work for $\lambda >\omega$ because, in the proof, we need colimits of chains of cofinality $\omega$. 
\item[(iii)] $E_\lambda$ is full if and only if $\Ho\ck_\lambda$ is weakly dense in~$\Ho\ck$.
\end{enumerate}

The homotopy category $\Ho\ck$ of any model category $\ck$ has weak colimits and weak limits. Weak colimits are constructed from coproducts and homotopy pushouts in the same way as colimits are constructed~from coproducts and pushouts.
A homotopy pushout of
\[
\xymatrix{PC & PA \ar[l]_-{Pg} \ar[r]^-{Pf} & PB}
\]
is a commutative diagram
\vspace*{-0.2cm}
\[
\xymatrix@C=3pc@R=3pc{
PA \ar[r]^{Pf_1} \ar[d]_{Pg_1} & PB_1 \ar[d]^{P\overline{g}}\\
PC_1 \ar [r]^{P\overline{f}} & PE}
\]
where $f=f_2\circ f_1$ and $g=g_2\circ g_1$ are factorizations of $f$ and $g$, respectively, into a cofibration followed by a trivial fibration, and
\[
\xymatrix@C=3pc@R=3pc{
A \ar[r]^{f_1} \ar[d]_{g_1} & B_1 \ar[d]^{\overline{g}}\\
C_1 \ar [r]^{\overline{f}} & E}
\]
is a pushout in $\ck$. The following definition is taken from \cite{BR}.

\begin{defn}
\label{nearly}
{\rm
A functor $H\colon\cc\to\cd$ will be called \textit{nearly full} if for each commutative triangle 
\[
\xymatrix@=3pc{
HA \ar[rr]^{Hh}
\ar[dr]_{f} && HC\\
& HB \ar[ur]_{Hg}}
\]
there is a morphism $\overline{f}\colon A\to B$ in $\cc$ such that $H\overline{f}=f$.
}
\end{defn}

\begin{prop}
\label{nearlybrown}
A strongly $\lambda$-combinatorial model category $\ck$ is $\lambda$-Brown on morphisms if and only if the functor $E_\lambda\colon\Ho\ck\to\Ind_\lambda\Ho\ck_\lambda$ is nearly full.
\end{prop}

\begin{proof}
Sufficiency is evident because any full functor is nearly full. Let $\ck$ be a strongly $\lambda$-combinatorial model category and assume 
that $E_\lambda$ is nearly full. Consider an object $K$ in $\ck$ and express it as a $\lambda$-filtered colimit $(\delta_d\colon Dd\to K)$ of its canonical diagram 
$D\colon\mathcal D\to\ck_\lambda$. This means that we have
\[
\xymatrix@C=3pc@R=3pc{
Dd \ar[d]_{u_e}  \ar[rd]^{v_d}& &\\
\coprod\limits_{e\colon\! d\to d'} Dd \ar@<0.8ex>[r]^p \ar@<-0.6ex>[r]_q &
\coprod\limits_d Dd \ar@<0.1ex>[r]^{g}&\underset{\ }{K}\\
Dd \ar[u]^-{u_e} \ar[r]_{De} &Dd' \ar[u]_-{v_{d'}}&}
\]
where $g$ is given by a pushout
\[
\xymatrix@C=3pc@R=3pc{
\coprod\limits_d Dd  \ar [r]^g  & K\\
\left(\coprod\limits_e Dd\right)\coprod \left(\coprod\limits_d Dd \right) \ar [u]^-{(p,\, \id)} \ar [r]_-{(q,\, \id)} &
\coprod\limits_d Dd. \ar [u]_g}
\]
If we replace the pushout above by a homotopy push\-out, we get $(\overline{\delta}_d\colon Dd\to\overline{K})$.
It is not a cocone in $\ck$ but $(P\overline{\delta}_d\colon PDd\to P\overline{K})$ is a standard weak colimit \cite{Christensen}
in~$\Ho\ck$, and there is a comparison morphism $t\colon\overline{K}\to K$ such that $t\circ\overline{\delta}_d=\delta_d$ for each $d$. Since 
$H_\lambda=\Ind_\lambda P_\lambda$ preserves $\lambda$-filtered colimits, there is a morphism $u\colon H_\lambda K\to H_\lambda\overline{K}$ such that
$u\circ H_\lambda\delta_d=H_\lambda\overline{\delta}_d$ for each $d$. Then
$H_\lambda t\circ u=\id$
because
\[
H_\lambda t\circ u\circ H_\lambda\delta_d=H_\lambda(t\circ\overline{\delta}_d)=H_\lambda\delta_d.
\]

Now, since $E_\lambda$ is nearly full, there is $\overline{u}\colon PK\to P\overline{K}$ such that $u=E_\lambda\overline{u}$. 

Consider a morphism  $h\colon H_\lambda K_1\to H_\lambda K_2$. Let $u_1$, $t_1$, $u_2$, $t_2$ be as $u$, $t$ above for $K_1$ and $K_2$. There is a cocone 
($\gamma_d\colon PD_1d\to P\overline{K}_2$) from $PD_1$ such that
\[
E_\lambda\gamma_d=u_2\circ h\circ H_\lambda\delta_{1d}\colon H_\lambda D_1d\longrightarrow H_\lambda \overline{K}_2
\] 
for each $d$ in $\mathcal D_1$. Thus there is a morphism 
$
\overline{h}\colon\overline{K}_1\to\overline{K}_2
$
such that 
$
\overline{h}\circ P\overline{\delta}_{1d}=\gamma_d
$
for each $d$ in $\mathcal D_1$. Hence
\[
E_\lambda\overline{h}\circ u_1\circ H_\lambda\delta_{1d}=E_\lambda\overline{h}\circ H_\lambda\overline{\delta}_{1d}=
E_\lambda\gamma_d=u_2\circ h\circ H_\lambda\delta_{1d}
\]
for each $d$ in $\mathcal D_1$. Thus
$
E_\lambda\overline{h}\circ u_1=u_2\circ h.
$
Putting $h'=Pt_2\circ\overline{h}\circ\overline{u}_1$, we obtain
\[
E_\lambda h'=E_\lambda(Pt_2\circ\overline{h})\circ u_1=H_\lambda t_2\circ u_2\circ h=h,
\]
which proves that $E_\lambda$ is full.
\qed
\end{proof}

\begin{rem}
\label{split}
{\rm
In Proposition~\ref{nearlybrown} it suffices to assume that $E_\lambda$ is full on split mono\-morphisms. This means that $h=\id$ in Definition~\ref{nearly}.
}
\end{rem}

The following result is in \cite[Proposition~6.4]{R2}.

\begin{prop}
\label{iso}
If $\ck$ is a combinatorial stable model category, then $E_\lambda$ reflects isomorphisms for arbitrarily large regular cardinals $\lambda$.
\end{prop}

\begin{rem}
\label{min}
{\rm
If $E_\lambda$ is full and reflects isomorphisms then each object of $\Ho\ck$ is a minimal weak colimit of its canonical diagram with respect to $\Ho\ck_\lambda$.
}
\end{rem}

One could ask if every combinatorial stable model category is $\lambda$-Brown for arbitrarily large regular cardinals $\lambda$, as discussed in \cite{R2} and \cite{R3}. 
This fact would have important consequences \cite{N2}, but it is unfortunately not true. The first counter\-example was given in~\cite{BG}, and in~\cite{BS} a large class was found of combinatorial stable model categories which are not $\lambda$-Brown for any~$\lambda$.
An obstruction theory for generalized Brown representability in triangulated categories was developed in \cite{MR}, with special focus on derived categories of rings.

\end{document}